\documentclass{amsart}
\usepackage{amsmath}
\usepackage{amsfonts}

\newtheorem{theorem}{Theorem}
\theoremstyle{plain}

\newtheorem{definition}{Definition}

\newtheorem{lemma}{Lemma}

\newtheorem{remark}{Remark}

\numberwithin{equation}{section}
\input{tcilatex}

\begin{document}
\title[ On Hermite-Hadamard Inequalities]{On Hermite-Hadamard Inequalities
for differentiable $\lambda $-preinvex functions via Riemann-Liouville
Fractional Integrals}
\author{Abdullah AKKURT}
\address{[Department of Mathematics, Faculty of Science and Arts, University
of Kahramanmara\c{s} S\"{u}t\c{c}\"{u} \.{I}mam, 46000, Kahramanmara\c{s},
Turkey}
\email{abdullahmat@gmail.com}
\author{M. Esra YILDIRIM}
\address{[Department of Mathematics, Faculty of Science, University of
Cumhuriyet, 58140, Sivas, Turkey}
\email{mesra@cumhuriyet.edu.tr}
\author{H\"{u}seyin YILDIRIM}
\address{[Department of Mathematics, Faculty of Science and Arts, University
of Kahramanmara\c{s} S\"{u}t\c{c}\"{u} \.{I}mam, 46000, Kahramanmara\c{s},
Turkey}
\email{hyildir@ksu.edu.tr}
\keywords{Fractional Hermite-Hadamard inequalities, preinvex functions,
Riemann-Liouville Fractional Integral. \textbf{}}
\thanks{M.E. Yildirim was partially supported by the Scientific and
Technological Research Council of Turkey (TUBITAK Programme 2228-B)}

\begin{abstract}
In this paper, we consider a new class of convex functions which is called $%
\lambda $-preinvex functions. We prove several Hermite--Hadamard type
inequalities for differentiable $\lambda $-preinvex functions via Fractional
Integrals. Some special cases are also discussed.
\end{abstract}

\maketitle

\section{INTRODUCTION}

The convexity property of a given function plays an important role in
obtaining integral inequalities. Proving inequalities for convex functions
has a long and rich history in mathamatics. Let $f:I\subseteq 
\mathbb{R}
\rightarrow 
\mathbb{R}
$ be a convex mapping defined on the interval $I$ of real numbers and $%
a,b\in I$ with $a<b$. The following inequality:

\begin{equation}
\begin{array}{c}
f\left( \dfrac{a+b}{2}\right) \leq \dfrac{1}{b-a}\int\limits_{a}^{b}f(x)dx%
\leq \dfrac{f(a)+f(b)}{2}.%
\end{array}
\label{1.1}
\end{equation}%
is known in the literature as Hermite-Hadamard inequality for convex
mappings. Note that some of the classical inequalities for means can be
derived from (\ref{1.1}) for appropriate particular selections of the
mapping $f.$ Both inequalities hold in the reversed direction if $f$ is
concave.

Over the last decade, this classical inequality has been improved and
generalized in a number of ways; there have been a large number of research
papers written on this subject, (see, \cite{Barani}-\cite{13}) and the
references therein.

A significant generalization of convex functions is that of invex functions
introduced by Hanson in \cite{hanson}. Ben-Israel and Mond \cite{BM}
introduced the concept of preinvex functions, which is a special case of
invexity. Noor \cite{Noor}-\cite{N3} has established some Hermite-Hadamard
type inequalities for preinvex and logpreinvex functions. In recent papers
Barani, Ghazanfari, and Dragomir in \cite{D1} presented some estimates of
the right hand side of a Hermite- Hadamard type inequality in which some
preinvex functions are involved. His class of nonconvex functions include
the classical convex functions and its various classes as special cases. For
some recent results related to this nonconvex functions, see the papers (%
\cite{Noor}-\cite{N3}, \cite{Pini}).

Now, we will give some definitions, lemmas and notations which we use later
in this work.

\begin{definition}
$\left( \text{\cite{3}}\right) $ Let $f\in L\left[ a,b\right] .$The\
Riemann-Liouville fractional integral $J_{a^{+}}^{\alpha }f$ and $%
J_{b^{-}}^{\alpha }f$ of order $\alpha >0$ with $a>0$ are defined by%
\begin{equation}
\begin{array}{ll}
J_{a^{+}}^{\alpha }f\left( x\right) =\dfrac{1}{\Gamma \left( \alpha \right) }%
\int_{a}^{x}(x-t)^{\alpha -1}f\left( t\right) dt & ,0\leq a<x\leq b \\ 
&  \\ 
J_{b^{-}}^{\alpha }f\left( x\right) =\dfrac{1}{\Gamma \left( \alpha \right) }%
\int_{x}^{b}(t-x)^{\alpha -1}f\left( t\right) dt & ,0\leq a<x\leq b%
\end{array}
\label{2.1}
\end{equation}%
Where $\Gamma $ is the gamma function.
\end{definition}

\begin{definition}
$\left( \text{\cite{9}}\right) $ The incomplete beta function is defined as
follows:%
\begin{equation}
\begin{array}{c}
B_{x}\left( a,b\right) =\int_{0}^{x}t^{a-1}\left( 1-t\right) ^{b-1}dt,%
\end{array}
\label{2.2}
\end{equation}%
Here $x\in \left[ 0,1\right] ,a,b>0$.
\end{definition}

\begin{definition}
$\left( \text{\cite{Abramovitz}}\right) \ $Gaussian hypergeometric function
defined by%
\begin{equation}
\begin{array}{c}
_{2}F_{1}(a,b;c;z)=\dfrac{1}{\beta (b,c-b)}\int_{0}^{1}t^{b-1}\left(
1-t\right) ^{c-b-1}\left( 1-zt\right) ^{-a}dt%
\end{array}
\label{2.21}
\end{equation}%
Here $c>b>0,\ \left\vert z\right\vert <1.$
\end{definition}

\begin{definition}
$\left( \text{\cite{10}}\right) $ A function $f:I\subseteq 
\mathbb{R}
\rightarrow 
\mathbb{R}
$ is said to belong to the class $MT\left( I\right) $ if $f$ is positive and 
$\forall x,y\in I$ and $t\in \left( 0,1\right) $ satisfies the inequality:%
\begin{equation}
\begin{array}{c}
f\left( tx+\left( 1-t\right) y\right) \leq \dfrac{\sqrt{t}}{2\sqrt{1-t}}%
f\left( x\right) +\dfrac{\sqrt{1-t}}{2\sqrt{t}}f\left( y\right) .%
\end{array}
\label{2.3}
\end{equation}
\end{definition}

\begin{definition}
A function $f:I\subseteq 
\mathbb{R}
\rightarrow 
\mathbb{R}
$ is said to a $\lambda -MT-$convex function or said to belong to the class $%
\lambda -MT\left( I\right) $ if $f$ is positive and $\forall x,y\in I,$ $%
\lambda \in \left( 0,\frac{1}{2}\right] $ and $t\in \left( 0,1\right) $
satisfies the inequality:%
\begin{equation}
\begin{array}{c}
f\left( tx+\left( 1-t\right) y\right) \leq \dfrac{\sqrt{t}}{2\sqrt{1-t}}%
f\left( x\right) +\dfrac{\left( 1-\lambda \right) \sqrt{1-t}}{2\lambda \sqrt{%
t}}f\left( y\right) .%
\end{array}
\label{2.5}
\end{equation}%
Meanwhile, Sarikaya et al. \cite{12} presented the following important
integral identity including the first-order derivative of $f$ to establish
many interesting Hermite--Hadamard-type inequalities for convexity functions
via Riemann--Liouville fractional integrals of the order $\alpha \in 
\mathbb{R}
^{+}$.
\end{definition}

\begin{lemma}
\label{LZ} Let $f:\left[ a,b\right] \rightarrow 
\mathbb{R}
$ be a once differentiable mapping on $\left( a,b\right) $ for $a<b$. If $%
f^{\prime }\in L\left[ a,b\right] ,$ there is a following equality for
fractional integrals%
\begin{equation}
\begin{array}{l}
\dfrac{f\left( a\right) +f\left( b\right) }{2}-\dfrac{\Gamma \left( \alpha
+1\right) }{2\left( b-a\right) ^{\alpha }}\left[ J_{a^{+}}^{\alpha }f\left(
b\right) +J_{b^{-}}^{\alpha }f\left( a\right) \right] \\ 
\\ 
\ \ \ \ \ \ =\dfrac{b-a}{2}\int\limits_{0}^{1}\left[ \left( 1-t\right)
^{\alpha }-t^{\alpha }\right] f^{\prime }\left( ta+\left( 1-t\right)
b\right) dt.%
\end{array}
\label{2.6}
\end{equation}%
Also, Wang et al. \cite{13} presented the following inequality.
\end{lemma}

\begin{lemma}
Let $f:\left[ a,b\right] \rightarrow 
\mathbb{R}
$ be a twice differentiable mapping on $\left( a,b\right) $ for $a<b$. If $%
f^{\prime \prime }\in L\left[ a,b\right] ,$ there is following equality for
fractional integrals%
\begin{equation}
\begin{array}{l}
\dfrac{f\left( a\right) +f\left( b\right) }{2}-\dfrac{\Gamma \left( \alpha
+1\right) }{2\left( b-a\right) ^{\alpha }}\left[ J_{a^{+}}^{\alpha }f\left(
b\right) +J_{b^{-}}^{\alpha }f\left( a\right) \right] \\ 
\\ 
=\dfrac{\left( b-a\right) ^{2}}{2}\int_{0}^{1}\left[ \dfrac{1-\left(
1-t\right) ^{\alpha +1}-t^{\alpha +1}}{\alpha +1}\right] f^{\prime \prime
}\left( ta+\left( 1-t\right) b\right) dt.%
\end{array}
\label{2.7}
\end{equation}%
In \cite{Dragomir}, Dragomir and Agarwal established the following result
connected with the right part of (\ref{1.1}):
\end{lemma}

\begin{theorem}
\label{TD} Let $f:I^{\circ }\subseteq 
\mathbb{R}
\rightarrow 
\mathbb{R}
$ be a differentiable mapping on $I^{\circ }$, $a,b\in I\ $with $a<b$. If $%
\left\vert f^{\prime }\right\vert $ is convex on $\left[ a,b\right] ,\ $then
the following inequalitiy holds:%
\begin{equation}
\begin{array}{l}
\dfrac{f\left( a\right) +f\left( b\right) }{2}-\dfrac{1}{b-a}%
\dint\limits_{a}^{b}f(x)dx\leq \dfrac{b-a}{8}\left[ \left\vert f^{\prime
}\left( a\right) \right\vert +\left\vert f^{\prime }\left( b\right)
\right\vert \right] .%
\end{array}
\label{2.70}
\end{equation}
\end{theorem}

\begin{lemma}
$\left( \text{\cite{Sumeyye}}\right) $ For any $A_{1}>A_{2}\geq 0$ and $%
p\geq 1,\ \left( A_{1}-A_{2}\right) ^{P}\leq A_{1}^{P}-A_{2}^{P}$.
\end{lemma}

\begin{lemma}
$\left( \text{\cite{14}}\right) $ For $t\in \left[ 0,1\right] ,$we have%
\begin{equation}
\begin{array}{cc}
\left( 1-t\right) ^{m}\leq 2^{1-m}-t^{m} & for\text{ }m\in \left[ 0,1\right]
, \\ 
\left( 1-t\right) ^{m}\geq 2^{1-m}-t^{m} & for\text{ }m\in \left[ 1,\infty
\right) .%
\end{array}
\label{2.71}
\end{equation}%
Let $%
\mathbb{R}
^{n}$ be Euclidian space and $K$ is said to a nonempty closed in $%
\mathbb{R}
^{n}.$ Let $f:K\rightarrow 
\mathbb{R}
\ $and $\eta :K\times K\rightarrow 
\mathbb{R}
$ be a continuous functions.
\end{lemma}

\begin{definition}
$($\cite{Noor}$)$ Let $u\in K.$ The set $K$ is said to be invex at $u$
according to $\eta $ if%
\begin{equation}
u+t\eta (v,u)\in K  \label{2.8}
\end{equation}%
for all $u,v\in K\ $and $t\in \lbrack 0,1].$
\end{definition}

\begin{definition}
\label{D7} Let $f:I\subseteq 
\mathbb{R}
\rightarrow 
\mathbb{R}
$ be a nonnegative function.A function $f$ on the set $K_{\eta }$ is said to
be $\lambda -$preinvex\ function according to bifunction $\eta $ and $%
\forall u,v\in I$, $t\in \left( 0,1\right) $, then 
\begin{equation}
\begin{array}{c}
f\left( u+t\eta \left( v,u\right) \right) \leq \dfrac{\sqrt{t}}{2\sqrt{1-t}}%
f\left( v\right) +\dfrac{\left( 1-\lambda \right) \sqrt{1-t}}{2\lambda \sqrt{%
t}}f\left( u\right) .%
\end{array}
\label{2.9}
\end{equation}
\end{definition}

\begin{remark}
In Definition \ref{D7}, if we choose $\lambda =\frac{1}{2},$ and $\eta
\left( v,u\right) =v-u.$ Definition \ref{D7} reduces to Definition $4;$%
\begin{equation*}
\begin{array}{c}
f\left( tv+\left( 1-t\right) u\right) \leq \dfrac{\sqrt{t}}{2\sqrt{1-t}}%
f\left( v\right) +\dfrac{\sqrt{1-t}}{2\sqrt{t}}f\left( u\right) .%
\end{array}%
\end{equation*}
\end{remark}

\begin{remark}
In Definition \ref{D7}, if we choose $\eta \left( v,u\right) =v-u.$
Definition \ref{D7} reduces to Definition $5;$%
\begin{equation*}
\begin{array}{c}
f\left( tv+\left( 1-t\right) u\right) \leq \frac{\sqrt{t}}{2\sqrt{1-t}}%
f\left( v\right) +\frac{\left( 1-\lambda \right) \sqrt{1-t}}{2\lambda \sqrt{t%
}}f\left( u\right) .%
\end{array}%
\end{equation*}%
Our goal in this paper is to state and prove the Hermite-Hadamard type
inequality for preinvex functions via Riemann-Liouville Fractional
Integrals. In order to achieve our goal, we first give two important lemmas
and then by using these identities we prove some integral inequalities.
\end{remark}

\section{Main Results}

We need the following lemma \cite{imdat}.

\begin{lemma}
\label{L1} Let $A\subseteq 
\mathbb{R}
$ be an open invex subset with respect to $\eta :A\times A\rightarrow 
\mathbb{R}
\ $and $a,b\in A\ $with $a<a+\eta (b,a).\ $If$\ f:A\rightarrow 
\mathbb{R}
\ $is a differentiable function such that $f^{\prime }\in L\left[ a,a+\eta
(b,a)\right] \ $then$,$ the following equality holds:%
\begin{equation}
\begin{array}{l}
\dfrac{f\left( a\right) +f\left( a+\eta (b,a)\right) }{2}-\dfrac{\Gamma
\left( \alpha +1\right) }{2\left( \eta (b,a)\right) ^{\alpha }}\left[
J_{a^{+}}^{\alpha }f\left( a+\eta (b,a)\right) +J_{\left( a+\eta
(b,a)\right) ^{-}}^{\alpha }f\left( a\right) \right]  \\ 
\\ 
=\dfrac{\eta (b,a)}{2}\int_{0}^{1}\left[ \left( 1-t\right) ^{\alpha
}-t^{\alpha }\right] f^{\prime }\left( a+(1-t)\eta (b,a)\right) dt.%
\end{array}
\label{3.1}
\end{equation}
\end{lemma}

\begin{proof}
Integrating by part and changing the variable of definite integral yield%
\begin{equation}
\begin{array}{l}
\int_{0}^{1}\left[ \left( 1-t\right) ^{\alpha }-t^{\alpha }\right] f^{\prime
}\left( a+\left( 1-t\right) \eta (b,a)\right) dt \\ 
\\ 
=\left. \left[ \left( 1-t\right) ^{\alpha }-t^{\alpha }\right] \dfrac{%
f\left( a+\left( 1-t\right) \eta (b,a)\right) }{-\eta (b,a)}\right\vert
_{0}^{1} \\ 
\\ 
-\dfrac{\alpha }{\eta (b,a)}\int_{0}^{1}\left[ \left( 1-t\right) ^{\alpha
-1}+t^{\alpha -1}\right] f\left( a+\left( 1-t\right) \eta (b,a)\right) dt \\ 
\\ 
=\dfrac{f\left( a\right) +f\left( a+\eta (b,a)\right) }{\eta (b,a)}-\dfrac{%
\alpha }{\eta (b,a)}\left[ \dfrac{1}{\left( \eta (b,a)\right) ^{\alpha }}%
\int_{a}^{a+\eta (b,a)}\left( a+\eta (b,a)-x\right) ^{\alpha -1}f\left(
x\right) dx\right. \\ 
\\ 
+\left. \dfrac{1}{\left( \eta (b,a)\right) ^{\alpha }}\int_{a}^{a+\eta
(b,a)}\left( x-a\right) ^{\alpha -1}f\left( x\right) dx\right] \\ 
\\ 
=\dfrac{f\left( a\right) +f\left( a+\eta (b,a)\right) }{\eta (b,a)}-\dfrac{%
\Gamma \left( \alpha +1\right) }{\left( \eta (b,a)\right) ^{\alpha +1}}\left[
J_{a^{+}}^{\alpha }f\left( a+\eta (b,a)\right) +J_{\left( a+\eta
(b,a)\right) ^{-}}^{\alpha }f\left( a\right) \right] .%
\end{array}
\label{3.2}
\end{equation}%
By multiplying the both sides of (\ref{3.2}) by $\dfrac{\eta (b,a)}{2},$ we
have:%
\begin{equation*}
\begin{array}{l}
\dfrac{f\left( a\right) +f\left( a+\eta (b,a)\right) }{2}-\dfrac{\Gamma
\left( \alpha +1\right) }{2\left( \eta (b,a)\right) ^{\alpha }}\left[
J_{a^{+}}^{\alpha }f\left( a+\eta (b,a)\right) +J_{\left( a+\eta
(b,a)\right) ^{-}}^{\alpha }f\left( a\right) \right] \\ 
\\ 
=\dfrac{\eta (b,a)}{2}\int_{0}^{1}\left[ \left( 1-t\right) ^{\alpha
}-t^{\alpha }\right] f^{\prime }\left( a+\left( 1-t\right) \eta (b,a)\right)
dt.%
\end{array}%
\end{equation*}%
Lemma \ref{L1} is thus proved.
\end{proof}

\begin{remark}
In Lemma \ref{L1}, if we choose $\eta \left( b,a\right) =b-a\ ,$ Lemma \ref%
{L1} reduces to Lemma$\ $\ref{LZ};
\end{remark}

\begin{theorem}
\label{T1} Let $A\subseteq 
\mathbb{R}
$ be an open invex subset with respect to $\eta :A\times A\rightarrow 
\mathbb{R}
\ $and $a,b\in A\ $with $a<a+\eta (b,a).\ $Suppose that $f:A\rightarrow 
\mathbb{R}
\ $is a differentiable function such that $f^{\prime }\in L\left[ a,a+\eta
(b,a)\right] .\ $If $\left\vert f^{\prime }\right\vert $ is $\lambda -$%
preinvex function on $\left[ a,a+\eta (b,a)\right] \ $then the following
inequality for fractional integrals with $\alpha >0$\ holds:%
\begin{equation*}
\begin{array}{l}
\left\vert \dfrac{f\left( a\right) +f\left( a+\eta (b,a)\right) }{2}-\dfrac{%
\Gamma \left( \alpha +1\right) }{2\left( \eta (b,a)\right) ^{\alpha }}\left[
J_{a^{+}}^{\alpha }f\left( a+\eta (b,a)\right) +J_{\left( a+\eta
(b,a)\right) ^{-}}^{\alpha }f\left( a\right) \right] \right\vert  \\ 
\\ 
\leq \dfrac{\eta (b,a)}{8}\left[ \left\vert f^{\prime }\left( a\right)
\right\vert +\dfrac{1-\lambda }{\lambda }\left\vert f^{\prime }\left(
b\right) \right\vert \right] \left\{ \dfrac{2\sqrt{\pi }\Gamma \left( \alpha
+\frac{3}{2}\right) }{\Gamma (\alpha +2)}-\dfrac{\sqrt{\pi }\Gamma \left(
\alpha +\frac{1}{2}\right) }{\Gamma (\alpha +2)}-4B_{\frac{1}{2}}\left(
\alpha +\frac{3}{2},\frac{1}{2}\right) \right.  \\ 
\\ 
+\dfrac{2^{-\alpha }\left( -\left( 4\alpha ^{2}+18\alpha +19\right)
\,_{2}F_{1}\left( 1,\alpha +2;\frac{1}{2};\frac{1}{2}\right) -2(\alpha
+2)\,_{2}F_{1}\left( 1,\alpha +2;-\frac{1}{2};\frac{1}{2}\right) \right) }{%
4\alpha ^{2}+8\alpha +3} \\ 
\\ 
\left. +\dfrac{2^{-\alpha }\left( -\alpha +2^{\alpha +\frac{1}{2}%
}\,_{2}F_{1}\left( -\frac{1}{2},\frac{1}{2}-\alpha ;\frac{1}{2};\frac{1}{2}%
\right) -1\right) }{\alpha (\alpha +1)}\right\} .%
\end{array}%
\end{equation*}
\end{theorem}

\begin{proof}
By using Definition \ref{D7} and Lemma \ref{L1}$,\ $we have:%
\begin{equation*}
\begin{array}{l}
\left\vert \dfrac{f\left( a\right) +f\left( a+\eta (b,a)\right) }{2}-\dfrac{%
\Gamma \left( \alpha +1\right) }{2\left( \eta (b,a)\right) ^{\alpha }}\left[
J_{a^{+}}^{\alpha }f\left( a+\eta (b,a)\right) +J_{\left( a+\eta
(b,a)\right) ^{-}}^{\alpha }f\left( a\right) \right] \right\vert \\ 
\\ 
\leq \frac{\eta (b,a)}{2}\int_{0}^{1}\left\vert \left( 1-t\right) ^{\alpha
}-t^{\alpha }\right\vert \left\vert f^{\prime }\left( a+\left( 1-t\right)
\eta (b,a)\right) \right\vert dt \\ 
\\ 
\leq \frac{\eta (b,a)}{2}\left[ \int_{0}^{\frac{1}{2}}\left[ \left(
1-t\right) ^{\alpha }-t^{\alpha }\right] \left\vert f^{\prime }\left(
a+\left( 1-t\right) \eta (b,a)\right) \right\vert dt\right. \\ 
\\ 
+\left. \int_{\frac{1}{2}}^{1}\left[ t^{\alpha }-\left( 1-t\right) ^{\alpha }%
\right] \left\vert f^{\prime }\left( a+\left( 1-t\right) \eta (b,a)\right)
\right\vert dt\right] \\ 
\\ 
\leq \frac{\eta (b,a)}{2}\left[ \int_{0}^{\frac{1}{2}}\left[ \left(
1-t\right) ^{\alpha }-t^{\alpha }\right] \left( \frac{\sqrt{t}}{2\sqrt{1-t}}%
\left\vert f^{\prime }\left( a\right) \right\vert +\frac{\left( 1-\lambda
\right) \sqrt{1-t}}{2\lambda \sqrt{t}}\left\vert f^{\prime }\left( b\right)
\right\vert \right) dt\right. \\ 
\\ 
+\left. \int_{\frac{1}{2}}^{1}\left[ t^{\alpha }-\left( 1-t\right) ^{\alpha }%
\right] \left( \frac{\sqrt{t}}{2\sqrt{1-t}}\left\vert f^{\prime }\left(
a\right) \right\vert +\frac{\left( 1-\lambda \right) \sqrt{1-t}}{2\lambda 
\sqrt{t}}\left\vert f^{\prime }\left( b\right) \right\vert \right) dt\right]
\\ 
\\ 
\leq \frac{\eta (b,a)}{2}\left[ \int_{0}^{1}\left[ \left( 1-t\right)
^{\alpha }+t^{\alpha }\right] \left( \frac{\sqrt{t}}{2\sqrt{1-t}}\left\vert
f^{\prime }\left( a\right) \right\vert +\frac{\left( 1-\lambda \right) \sqrt{%
1-t}}{2\lambda \sqrt{t}}\left\vert f^{\prime }\left( b\right) \right\vert
\right) dt\right. \\ 
\\ 
\leq \dfrac{\eta (b,a)}{8}\left[ \left\vert f^{\prime }\left( a\right)
\right\vert +\dfrac{1-\lambda }{\lambda }\left\vert f^{\prime }\left(
b\right) \right\vert \right] \left\{ \dfrac{2\sqrt{\pi }\Gamma \left( \alpha
+\frac{3}{2}\right) }{\Gamma (\alpha +2)}-\dfrac{\sqrt{\pi }\Gamma \left(
\alpha +\frac{1}{2}\right) }{\Gamma (\alpha +2)}-4B_{\frac{1}{2}}\left(
\alpha +\frac{3}{2},\frac{1}{2}\right) \right. \\ 
\\ 
+\dfrac{2^{-\alpha }\left( -\left( 4\alpha ^{2}+18\alpha +19\right)
\,_{2}F_{1}\left( 1,\alpha +2;\frac{1}{2};\frac{1}{2}\right) -2(\alpha
+2)\,_{2}F_{1}\left( 1,\alpha +2;-\frac{1}{2};\frac{1}{2}\right) \right) }{%
4\alpha ^{2}+8\alpha +3} \\ 
\\ 
\left. +\dfrac{2^{-\alpha }\left( -\alpha +2^{\alpha +\frac{1}{2}%
}\,_{2}F_{1}\left( -\frac{1}{2},\frac{1}{2}-\alpha ;\frac{1}{2};\frac{1}{2}%
\right) -1\right) }{\alpha (\alpha +1)}\right\} .%
\end{array}%
\end{equation*}%
The proof is done.
\end{proof}

\begin{remark}
If we take $\eta \left( b,a\right) =b-a,\ \lambda =\dfrac{1}{2}\ $and $%
\alpha =1\ $in Theorem \ref{T1}$,$ Theorem \ref{T1} reduces to Theorem$\ $%
\ref{TD};
\end{remark}

\begin{theorem}
\label{T2}\ Let $A\subseteq 
\mathbb{R}
$ be an open invex subset with respect to $\eta :A\times A\rightarrow 
\mathbb{R}
\ $and $a,b\in A\ $with $a<a+\eta (b,a).\ $Suppose that $f:A\rightarrow 
\mathbb{R}
\ $is a differentiable function such that $f^{\prime }\in L\left[ a,a+\eta
(b,a)\right] .\ $If $\left\vert f^{\prime }\right\vert ^{q}$ is $\lambda -$%
preinvex function on $\left[ a,a+\eta (b,a)\right] \ $for some fixed $q>1$
then the following inequality holds:%
\begin{equation*}
\begin{array}{l}
\left\vert \dfrac{f\left( a\right) +f\left( a+\eta (b,a)\right) }{2}-\dfrac{%
\Gamma \left( \alpha +1\right) }{2\left( \eta (b,a)\right) ^{\alpha }}\left[
J_{a^{+}}^{\alpha }f\left( a+\eta (b,a)\right) +J_{\left( a+\eta
(b,a)\right) ^{-}}^{\alpha }f\left( a\right) \right] \right\vert  \\ 
\\ 
\leq \dfrac{\eta (b,a)}{2}\left( \dfrac{\pi }{4}\right) ^{\frac{1}{q}}\left( 
\dfrac{2-2^{1-\alpha p}}{p\alpha +1}\right) ^{\frac{1}{p}}\left[ \left\vert
f^{\prime }\left( a\right) \right\vert ^{q}+\left( \dfrac{1-\lambda }{%
\lambda }\right) \left\vert f^{\prime }\left( b\right) \right\vert ^{q}%
\right] ^{\frac{1}{q}}%
\end{array}%
\end{equation*}%
where $\alpha \in \left[ 0,1\right] \ $and $\frac{1}{p}+\frac{1}{q}=1.$
\end{theorem}

\begin{proof}
By using Definition \ref{D7}, Lemma \ref{L1} and H\"{o}lder's inequality, we
have:%
\begin{equation*}
\begin{array}{l}
\left\vert \dfrac{f\left( a\right) +f\left( a+\eta (b,a)\right) }{2}-\dfrac{%
\Gamma \left( \alpha +1\right) }{2\left( \eta (b,a)\right) ^{\alpha }}\left[
J_{a^{+}}^{\alpha }f\left( a+\eta (b,a)\right) +J_{\left( a+\eta
(b,a)\right) ^{-}}^{\alpha }f\left( a\right) \right] \right\vert \\ 
\\ 
\leq \frac{\eta (b,a)}{2}\int_{0}^{1}\left\vert \left( 1-t\right) ^{\alpha
}-t^{\alpha }\right\vert \left\vert f^{\prime }\left( a+\left( 1-t\right)
\eta (b,a)\right) \right\vert dt \\ 
\\ 
\leq \frac{\eta (b,a)}{2}\left( \int_{0}^{1}\left\vert \left( 1-t\right)
^{\alpha }-t^{\alpha }\right\vert ^{p}dt\right) ^{\frac{1}{p}}\left(
\int_{0}^{1}\left\vert f^{\prime }\left( a+\left( 1-t\right) \eta
(b,a)\right) \right\vert ^{q}dt\right) ^{\frac{1}{q}} \\ 
\\ 
\leq \frac{\eta (b,a)}{2}\left( \int_{0}^{1}\left\vert \left( 1-t\right)
^{\alpha }-t^{\alpha }\right\vert ^{p}dt\right) ^{\frac{1}{p}} \\ 
\\ 
\times \left( \int_{0}^{1}\left( \frac{\sqrt{t}}{2\sqrt{1-t}}\left\vert
f^{\prime }\left( a\right) \right\vert ^{q}+\frac{\left( 1-\lambda \right) 
\sqrt{1-t}}{2\lambda \sqrt{t}}\left\vert f^{\prime }\left( b\right)
\right\vert ^{q}\right) dt\right) ^{\frac{1}{q}} \\ 
\\ 
\leq \frac{\eta (b,a)}{2}\left[ \frac{\pi }{4}\left\vert f^{\prime }\left(
a\right) \right\vert ^{q}+\frac{\pi }{4}\left( \frac{1-\lambda }{\lambda }%
\right) \left\vert f^{\prime }\left( b\right) \right\vert ^{q}\right] ^{%
\frac{1}{q}} \\ 
\\ 
\times \left( \int_{0}^{\frac{1}{2}}\left[ \left( 1-t\right) ^{\alpha
p}-t^{\alpha p}\right] dt+\int_{\frac{1}{2}}^{1}\left[ t^{\alpha p}-\left(
1-t\right) ^{\alpha p}\right] dt\right) ^{\frac{1}{p}} \\ 
\\ 
\leq \dfrac{\eta (b,a)}{2}\left[ \left\vert f^{\prime }\left( a\right)
\right\vert ^{q}+\frac{1-\lambda }{\lambda }\left\vert f^{\prime }\left(
b\right) \right\vert ^{q}\right] ^{\frac{1}{q}}\left( \frac{\pi }{4}\right)
^{\frac{1}{q}}\left( \dfrac{2-2^{1-\alpha p}}{\alpha p+1}\right) ^{\frac{1}{p%
}}.%
\end{array}%
\end{equation*}
The proof is done.
\end{proof}

\begin{remark}
In Theorem \ref{T2}, if we choose $\eta \left( b,a\right) =b-a,$ then we
have;%
\begin{equation*}
\begin{array}{l}
\left\vert \dfrac{f\left( a\right) +f\left( b\right) }{2}-\dfrac{\Gamma
\left( \alpha +1\right) }{2\left( b-a\right) ^{\alpha }}\left[
J_{a^{+}}^{\alpha }f\left( b\right) +J_{b^{-}}^{\alpha }f\left( a\right) %
\right] \right\vert \\ 
\\ 
\leq \dfrac{b-a}{2}\left( \dfrac{\pi }{4}\right) ^{\frac{1}{q}}\left( \dfrac{%
2-2^{1-\alpha p}}{p\alpha +1}\right) ^{\frac{1}{p}}\left[ \left\vert
f^{\prime }\left( a\right) \right\vert ^{q}+\left( \dfrac{1-\lambda }{%
\lambda }\right) \left\vert f^{\prime }\left( b\right) \right\vert ^{q}%
\right] ^{\frac{1}{q}}.%
\end{array}%
\end{equation*}
\end{remark}

\begin{remark}
In Theorem \ref{T2}, if we choose $\eta \left( b,a\right) =b-a\ $and $\alpha
=1,$ then we have;%
\begin{equation*}
\begin{array}{l}
\left\vert \dfrac{f\left( a\right) +f\left( b\right) }{2}-\dfrac{1}{b-a}%
\dint\limits_{a}^{b}f(x)dx\right\vert \\ 
\\ 
\leq \dfrac{b-a}{2}\left[ \dfrac{\pi }{4}\left\vert f^{\prime }\left(
a\right) \right\vert ^{q}+\dfrac{\pi }{4}\left( \dfrac{1-\lambda }{\lambda }%
\right) \left\vert f^{\prime }\left( b\right) \right\vert ^{q}\right] ^{%
\frac{1}{q}}\left( \dfrac{2-2^{1-p}}{p+1}\right) ^{\frac{1}{p}}.%
\end{array}%
\end{equation*}
\end{remark}

\begin{remark}
In Theorem \ref{T2}, if we choose $\eta \left( b,a\right) =b-a,\ \lambda
=1/2\ $and $\alpha =1,$ then we have;%
\begin{equation*}
\begin{array}{l}
\left\vert \dfrac{f\left( a\right) +f\left( b\right) }{2}-\dfrac{1}{b-a}%
\dint\limits_{a}^{b}f(x)dx\right\vert \\ 
\\ 
\leq \dfrac{b-a}{8}\pi \left[ \left\vert f^{\prime }\left( a\right)
\right\vert ^{q}+\left\vert f^{\prime }\left( b\right) \right\vert ^{q}%
\right] ^{\frac{1}{q}}\left( \dfrac{2-2^{1-p}}{p+1}\right) ^{\frac{1}{p}}.%
\end{array}%
\end{equation*}
\end{remark}

\begin{theorem}
\label{T3} Let $A\subseteq 
\mathbb{R}
$ be an open invex subset with respect to $\eta :A\times A\rightarrow 
\mathbb{R}
\ $and $a,b\in A\ $with $a<a+\eta (b,a).\ $Suppose that $f:A\rightarrow 
\mathbb{R}
\ $is a differentiable function such that $f^{\prime }\in L\left[ a,a+\eta
(b,a)\right] .\ $If $\left\vert f^{\prime }\right\vert ^{q}$ is $\lambda -$%
preinvex function on $\left[ a,a+\eta (b,a)\right] \ $for some fixed $q>1$
then the following inequality holds:%
\begin{equation*}
\begin{array}{l}
\left\vert \dfrac{f\left( a\right) +f\left( a+\eta (b,a)\right) }{2}-\dfrac{%
\Gamma \left( \alpha +1\right) }{2\left( \eta (b,a)\right) ^{\alpha }}\left[
J_{a^{+}}^{\alpha }f\left( a+\eta (b,a)\right) +J_{\left( a+\eta
(b,a)\right) ^{-}}^{\alpha }f\left( a\right) \right] \right\vert  \\ 
\\ 
\leq \left( \frac{1-2^{-\alpha }}{\alpha +1}\right) ^{\frac{q-1}{q}}\dfrac{%
\eta (b,a)}{2^{1+1/q}}\left[ \left\vert f^{\prime }\left( a\right)
\right\vert ^{q}+\dfrac{1-\lambda }{\lambda }\left\vert f^{\prime }\left(
b\right) \right\vert ^{q}\right]  \\ 
\\ 
\times \left\{ \dfrac{2\sqrt{\pi }\Gamma \left( \alpha +\frac{3}{2}\right) }{%
\Gamma (\alpha +2)}-\dfrac{\sqrt{\pi }\Gamma \left( \alpha +\frac{1}{2}%
\right) }{\Gamma (\alpha +2)}-4B_{\frac{1}{2}}\left( \alpha +\frac{3}{2},%
\frac{1}{2}\right) \right.  \\ 
\\ 
+\dfrac{2^{-\alpha }\left( -\left( 4\alpha ^{2}+18\alpha +19\right)
\,_{2}F_{1}\left( 1,\alpha +2;\frac{1}{2};\frac{1}{2}\right) -2(\alpha
+2)\,_{2}F_{1}\left( 1,\alpha +2;-\frac{1}{2};\frac{1}{2}\right) \right) }{%
4\alpha ^{2}+8\alpha +3} \\ 
\\ 
\left. +\dfrac{2^{-\alpha }\left( -\alpha +2^{\alpha +\frac{1}{2}%
}\,_{2}F_{1}\left( -\frac{1}{2},\frac{1}{2}-\alpha ;\frac{1}{2};\frac{1}{2}%
\right) -1\right) }{\alpha (\alpha +1)}\right\} ^{1/q}.%
\end{array}%
\end{equation*}%
where $\alpha \in \left[ 0,1\right] \ $and $\frac{1}{p}+\frac{1}{q}=1.$
\end{theorem}

\begin{proof}
By using Definition \ref{D7}, Lemma \ref{L1} and Power Mean inequality, we
have:%
\begin{equation*}
\begin{array}{l}
\left\vert \dfrac{f\left( a\right) +f\left( a+\eta (b,a)\right) }{2}-\dfrac{%
\Gamma \left( \alpha +1\right) }{2\left( \eta (b,a)\right) ^{\alpha }}\left[
J_{a^{+}}^{\alpha }f\left( a+\eta (b,a)\right) +J_{\left( a+\eta
(b,a)\right) ^{-}}^{\alpha }f\left( a\right) \right] \right\vert \\ 
\\ 
\leq \dfrac{\eta (b,a)}{2}\int_{0}^{1}\left\vert \left( 1-t\right) ^{\alpha
}-t^{\alpha }\right\vert \left\vert f^{\prime }\left( a+\left( 1-t\right)
\eta (b,a)\right) \right\vert dt \\ 
\\ 
\leq \dfrac{\eta (b,a)}{2}\left( \int_{0}^{1}\left\vert \left( 1-t\right)
^{\alpha }-t^{\alpha }\right\vert dt\right) ^{1-\frac{1}{q}} \\ 
\\ 
\times \left( \int_{0}^{1}\left\vert \left( 1-t\right) ^{\alpha }-t^{\alpha
}\right\vert \left\vert f^{\prime }\left( a+\left( 1-t\right) \eta
(b,a)\right) \right\vert ^{q}dt\right) ^{\frac{1}{q}} \\ 
\\ 
\leq \dfrac{\eta (b,a)}{2}\left( \int_{0}^{\frac{1}{2}}\left[ \left(
1-t\right) ^{\alpha }-t^{\alpha }\right] dt+\int_{\frac{1}{2}}^{1}\left[
t^{\alpha }-\left( 1-t\right) ^{\alpha }\right] dt\right) ^{1-\frac{1}{q}}
\\ 
\\ 
\times \left( \int_{0}^{1}\left\vert \left( 1-t\right) ^{\alpha }-t^{\alpha
}\right\vert \left\vert f^{\prime }\left( a+\left( 1-t\right) \eta
(b,a)\right) \right\vert ^{q}dt\right) ^{\frac{1}{q}}%
\end{array}%
\end{equation*}%
\begin{equation*}
\begin{array}{l}
\leq \dfrac{\eta (b,a)}{2}\left( \dfrac{2-2^{1-\alpha }}{\alpha +1}\right) ^{%
\frac{q-1}{q}}\left[ \int_{0}^{\frac{1}{2}}\left[ \left( 1-t\right) ^{\alpha
}-t^{\alpha }\right] \left( \dfrac{\sqrt{t}}{2\sqrt{1-t}}\left\vert
f^{\prime }\left( a\right) \right\vert ^{q}+\dfrac{\left( 1-\lambda \right) 
\sqrt{1-t}}{2\lambda \sqrt{t}}\left\vert f^{\prime }\left( b\right)
\right\vert ^{q}\right) dt\right. \\ 
\\ 
+\left. \int_{\frac{1}{2}}^{1}\left[ t^{\alpha }-\left( 1-t\right) ^{\alpha }%
\right] \left( \dfrac{\sqrt{t}}{2\sqrt{1-t}}\left\vert f^{\prime }\left(
a\right) \right\vert ^{q}+\dfrac{\left( 1-\lambda \right) \sqrt{1-t}}{%
2\lambda \sqrt{t}}\left\vert f^{\prime }\left( b\right) \right\vert
^{q}\right) dt\right] ^{\frac{1}{q}} \\ 
\\ 
\leq \left( \dfrac{1-2^{-\alpha }}{\alpha +1}\right) ^{\frac{q-1}{q}}\dfrac{%
\eta (b,a)}{2^{1+1/q}}\left[ \left\vert f^{\prime }\left( a\right)
\right\vert ^{q}+\dfrac{1-\lambda }{\lambda }\left\vert f^{\prime }\left(
b\right) \right\vert ^{q}\right] \\ 
\\ 
\times \left\{ \dfrac{2\sqrt{\pi }\Gamma \left( \alpha +\frac{3}{2}\right) }{%
\Gamma (\alpha +2)}-\dfrac{\sqrt{\pi }\Gamma \left( \alpha +\frac{1}{2}%
\right) }{\Gamma (\alpha +2)}-4B_{\frac{1}{2}}\left( \alpha +\frac{3}{2},%
\frac{1}{2}\right) \right. \\ 
\\ 
+\dfrac{2^{-\alpha }\left( -\left( 4\alpha ^{2}+18\alpha +19\right)
\,_{2}F_{1}\left( 1,\alpha +2;\frac{1}{2};\frac{1}{2}\right) -2(\alpha
+2)\,_{2}F_{1}\left( 1,\alpha +2;-\frac{1}{2};\frac{1}{2}\right) \right) }{%
4\alpha ^{2}+8\alpha +3} \\ 
\\ 
\left. +\dfrac{2^{-\alpha }\left( -\alpha +2^{\alpha +\frac{1}{2}%
}\,_{2}F_{1}\left( -\frac{1}{2},\frac{1}{2}-\alpha ;\frac{1}{2};\frac{1}{2}%
\right) -1\right) }{\alpha (\alpha +1)}\right\} ^{1/q}.%
\end{array}%
\end{equation*}%
The proof is done.
\end{proof}

\begin{remark}
In Theorem \ref{T3}, if we choose $\eta \left( b,a\right) =b-a\ $and $\alpha
=1,$ then we have;%
\begin{equation*}
\begin{array}{l}
\left\vert \dfrac{f\left( a\right) +f\left( b\right) }{2}-\dfrac{1}{b-a}%
\dint\limits_{a}^{b}f(x)dx\right\vert \leq \left( \dfrac{1}{4}\right) ^{%
\frac{q-1}{q}}\dfrac{b-a}{2^{1+1/q}}\left[ \left\vert f^{\prime }\left(
a\right) \right\vert ^{q}+\dfrac{1-\lambda }{\lambda }\left\vert f^{\prime
}\left( b\right) \right\vert ^{q}\right] .%
\end{array}%
\end{equation*}
\end{remark}

\begin{remark}
In Theorem \ref{T3}, if we choose $\eta \left( b,a\right) =b-a,\ \lambda
=1/2\ $and $\alpha =1,$ then we have;%
\begin{equation*}
\begin{array}{l}
\left\vert \dfrac{f\left( a\right) +f\left( b\right) }{2}-\dfrac{1}{b-a}%
\dint\limits_{a}^{b}f(x)dx\right\vert \leq 2^{\frac{1}{q}}\dfrac{b-a}{8}%
\left[ \left\vert f^{\prime }\left( a\right) \right\vert ^{q}+\left\vert
f^{\prime }\left( b\right) \right\vert ^{q}\right] .%
\end{array}%
\end{equation*}
\end{remark}

\begin{lemma}
\label{L2} Let $A\subseteq 
\mathbb{R}
$ be an open invex subset with respect to $\eta :A\times A\rightarrow 
\mathbb{R}
\ $and $a,b\in A\ $with $a<a+\eta (b,a).\ $If$\ f:A\rightarrow 
\mathbb{R}
\ $is a differentiable function such that $f^{\prime \prime }\in L\left[
a,a+\eta (b,a)\right] \ $then$,$ the following equality holds:%
\begin{equation}
\begin{array}{l}
\left\vert \dfrac{f\left( a\right) +f\left( a+\eta (b,a)\right) }{2}-\dfrac{%
\Gamma \left( \alpha +1\right) }{2\left( \eta (b,a)\right) ^{\alpha }}\left[
J_{a^{+}}^{\alpha }f\left( a+\eta (b,a)\right) +J_{\left( a+\eta
(b,a)\right) ^{-}}^{\alpha }f\left( a\right) \right] \right\vert  \\ 
\\ 
=\dfrac{\left( \eta (b,a)\right) ^{2}}{2(\alpha +1)}\int_{0}^{1}\left[
1-\left( 1-t\right) ^{\alpha +1}-t^{\alpha +1}\right] f^{\prime \prime
}\left( a+\left( 1-t\right) \eta (b,a)\right) dt.%
\end{array}
\label{3.4}
\end{equation}
\end{lemma}

\begin{proof}
Integrating by part and changing the variable of definite integral yield%
\begin{equation}
\begin{array}{l}
\int_{0}^{1}\left[ \dfrac{1-\left( 1-t\right) ^{\alpha +1}-t^{\alpha +1}}{%
\alpha +1}\right] f^{\prime \prime }\left( a+\left( 1-t\right) \eta
(b,a)\right) dt \\ 
\\ 
=\left. -\dfrac{\left( 1-\left( 1-t\right) ^{\alpha +1}-t^{\alpha +1}\right)
f^{\prime }\left( a+\left( 1-t\right) \eta (b,a)\right) }{\left( \alpha
+1\right) \eta (b,a)}\right\vert _{0}^{1} \\ 
\\ 
+\dfrac{1}{\eta (b,a)}\int_{0}^{1}\left[ \left( 1-t\right) ^{\alpha
}-t^{\alpha }\right] f^{\prime }\left( a+\left( 1-t\right) \eta (b,a)\right)
dt \\ 
\\ 
=\dfrac{1}{\eta (b,a)}\int_{0}^{1}\left[ \left( 1-t\right) ^{\alpha
}-t^{\alpha }\right] f^{\prime }\left( a+\left( 1-t\right) \eta (b,a)\right)
dt.%
\end{array}
\label{3.5}
\end{equation}%
Motivated by Lemma$\ $\ref{L1}, then:%
\begin{equation*}
\begin{array}{l}
\dfrac{1}{\eta (b,a)}\left( \int_{0}^{1}\left[ \left( 1-t\right) ^{\alpha
}-t^{\alpha }\right] f^{\prime }\left( a+\left( 1-t\right) \eta (b,a)\right)
dt\right) \\ 
\\ 
=\dfrac{f\left( a\right) +f\left( a+\eta (b,a)\right) }{\left( \eta
(b,a)\right) ^{2}}-\dfrac{\Gamma \left( \alpha +1\right) }{\left( \eta
(b,a)\right) ^{\alpha +2}}\left[ J_{a^{+}}^{\alpha }f\left( a+\eta
(b,a)\right) +J_{\left( a+\eta (b,a)\right) ^{-}}^{\alpha }f\left( a\right) %
\right] .%
\end{array}%
\end{equation*}%
By multipling the both sides of (\ref{3.5}) by $\dfrac{\left( \eta
(b,a)\right) ^{2}}{2}$, we have:%
\begin{equation*}
\begin{array}{l}
\left\vert \dfrac{f\left( a\right) +f\left( a+\eta (b,a)\right) }{2}-\dfrac{%
\Gamma \left( \alpha +1\right) }{2\left( \eta (b,a)\right) ^{\alpha }}\left[
J_{a^{+}}^{\alpha }f\left( a+\eta (b,a)\right) +J_{\left( a+\eta
(b,a)\right) ^{-}}^{\alpha }f\left( a\right) \right] \right\vert \\ 
\\ 
=\dfrac{\left( \eta (b,a)\right) ^{2}}{2}\int_{0}^{1}\left[ \dfrac{1-\left(
1-t\right) ^{\alpha +1}-t^{\alpha +1}}{\alpha +1}\right] f^{\prime \prime
}\left( a+\left( 1-t\right) \eta (b,a)\right) dt%
\end{array}%
\end{equation*}%
The proof is done.
\end{proof}

\begin{remark}
In Lemma \ref{L2}, $\eta \left( b,a\right) =b-a.$ Lemma \ref{L2} reduces to
Lemma $2;$%
\begin{equation*}
\begin{array}{l}
\dfrac{f\left( a\right) +f\left( b\right) }{2}-\dfrac{\Gamma \left( \alpha
+1\right) }{2\left( b-a\right) ^{\alpha }}\left[ J_{a^{+}}^{\alpha }f\left(
b\right) +J_{b^{-}}^{\alpha }f\left( a\right) \right] \\ 
\\ 
=\dfrac{\left( b-a\right) ^{2}}{2}\int_{0}^{1}\left[ \dfrac{1-\left(
1-t\right) ^{\alpha +1}-t^{\alpha +1}}{\alpha +1}\right] f^{\prime \prime
}\left( ta+\left( 1-t\right) b\right) dt.%
\end{array}%
\end{equation*}
\end{remark}

\begin{theorem}
\label{T4} Let $A\subseteq 
\mathbb{R}
$ be an open invex subset with respect to $\eta :A\times A\rightarrow 
\mathbb{R}
\ $and $a,b\in A\ $with $a<a+\eta (b,a).\ $Suppose that $f:A\rightarrow 
\mathbb{R}
\ $is a differentiable function such that $f^{\prime \prime }\in L\left[
a,a+\eta (b,a)\right] .\ $If $\left\vert f^{\prime \prime }\right\vert $ is $%
\lambda -$preinvex function on $\left[ a,a+\eta (b,a)\right] \ $then the
following inequality for fractional integrals with $\alpha >0$\ holds:%
\begin{equation*}
\begin{array}{l}
\left\vert \dfrac{f\left( a\right) +f\left( a+\eta (b,a)\right) }{2}-\dfrac{%
\Gamma \left( \alpha +1\right) }{2\left( \eta (b,a)\right) ^{\alpha }}\left[
J_{a^{+}}^{\alpha }f\left( a+\eta (b,a)\right) +J_{\left( a+\eta
(b,a)\right) ^{-}}^{\alpha }f\left( a\right) \right] \right\vert  \\ 
\\ 
\leq \dfrac{\left( \eta (b,a)\right) ^{2}}{4\left( \alpha +1\right) }\left( 
\dfrac{\pi }{2}-\dfrac{\sqrt{\pi }\Gamma \left( \alpha +\frac{3}{2}\right) }{%
\Gamma \left( \alpha +2\right) }\right) \left\{ \left\vert f^{\prime \prime
}\left( a\right) \right\vert +\left( \dfrac{1-\lambda }{\lambda }\right)
\left\vert f^{\prime \prime }\left( b\right) \right\vert \right\} .%
\end{array}%
\end{equation*}
\end{theorem}

\begin{proof}
By using Definition \ref{D7} and Lemma \ref{L2}, we have:%
\begin{equation*}
\begin{array}{l}
\left\vert \dfrac{f\left( a\right) +f\left( a+\eta (b,a)\right) }{2}-\dfrac{%
\Gamma \left( \alpha +1\right) }{2\left( \eta (b,a)\right) ^{\alpha }}\left[
J_{a^{+}}^{\alpha }f\left( a+\eta (b,a)\right) +J_{\left( a+\eta
(b,a)\right) ^{-}}^{\alpha }f\left( a\right) \right] \right\vert \\ 
\\ 
\leq \dfrac{\left( \eta (b,a)\right) ^{2}}{2}\int_{0}^{1}\left\vert \dfrac{%
1-\left( 1-t\right) ^{\alpha +1}-t^{\alpha +1}}{\alpha +1}\right\vert
\left\vert f^{\prime \prime }\left( a+\left( 1-t\right) \eta (b,a)\right)
\right\vert dt \\ 
\\ 
\leq \dfrac{\left( \eta (b,a)\right) ^{2}}{2\left( \alpha +1\right) }%
\int_{0}^{1}\left\vert 1-\left( 1-t\right) ^{\alpha +1}-t^{\alpha
+1}\right\vert \left( \dfrac{\sqrt{t}}{2\sqrt{1-t}}\left\vert f^{\prime
\prime }\left( a\right) \right\vert +\dfrac{\left( 1-\lambda \right) \sqrt{%
1-t}}{2\lambda \sqrt{t}}\left\vert f^{\prime \prime }\left( b\right)
\right\vert \right) dt \\ 
\\ 
\leq \dfrac{\left( \eta (b,a)\right) ^{2}}{2\left( \alpha +1\right) }\left\{ 
\dfrac{\left\vert f^{\prime \prime }\left( a\right) \right\vert }{2}%
\int_{0}^{1}\left( 1-\left( 1-t\right) ^{\alpha +1}-t^{\alpha +1}\right) 
\dfrac{\sqrt{t}}{\sqrt{1-t}}dt\right. \\ 
\\ 
\left. +\dfrac{1-\lambda }{\lambda }\dfrac{\left\vert f^{\prime \prime
}\left( b\right) \right\vert }{2}\int_{0}^{1}\left( 1-\left( 1-t\right)
^{\alpha +1}-t^{\alpha +1}\right) \dfrac{\sqrt{1-t}}{\sqrt{t}}dt\right\} \\ 
\\ 
\leq \dfrac{\left( \eta (b,a)\right) ^{2}}{4\left( \alpha +1\right) }\left( 
\dfrac{\pi }{2}-\dfrac{\sqrt{\pi }\Gamma \left( \alpha +\frac{3}{2}\right) }{%
\Gamma \left( \alpha +2\right) }\right) \left\{ \left\vert f^{\prime \prime
}\left( a\right) \right\vert +\left( \dfrac{1-\lambda }{\lambda }\right)
\left\vert f^{\prime \prime }\left( b\right) \right\vert \right\} .%
\end{array}%
\end{equation*}%
The proof is done.
\end{proof}

\begin{remark}
In Theorem \ref{T4}, if we take $\eta \left( b,a\right) =b-a,\ $we have$;$%
\begin{equation*}
\begin{array}{l}
\left\vert \dfrac{f\left( a\right) +f\left( b\right) }{2}-\dfrac{\Gamma
\left( \alpha +1\right) }{2\left( b-a\right) ^{\alpha }}\left[
J_{a^{+}}^{\alpha }f\left( b\right) +J_{b^{-}}^{\alpha }f\left( a\right) %
\right] \right\vert \\ 
\\ 
\leq \dfrac{\left( b-a\right) ^{2}}{4\left( \alpha +1\right) }\left( \dfrac{%
\pi }{2}-\dfrac{\sqrt{\pi }\Gamma \left( \alpha +\frac{3}{2}\right) }{\Gamma
\left( \alpha +2\right) }\right) \left\{ \left\vert f^{\prime \prime }\left(
a\right) \right\vert +\left( \dfrac{1-\lambda }{\lambda }\right) \left\vert
f^{\prime \prime }\left( b\right) \right\vert \right\} .%
\end{array}%
\end{equation*}
\end{remark}

\begin{remark}
In Theorem \ref{T4}, if we take $\eta \left( b,a\right) =b-a\ $and $\alpha
=1,\ $we have$;$%
\begin{equation*}
\begin{array}{l}
\left\vert \dfrac{f\left( a\right) +f\left( b\right) }{2}-\dfrac{1}{b-a}%
\dint\limits_{a}^{b}f(x)dx\right\vert \leq \dfrac{\pi \left( b-a\right) ^{2}%
}{64}\left\{ \left\vert f^{\prime \prime }\left( a\right) \right\vert
+\left( \dfrac{1-\lambda }{\lambda }\right) \left\vert f^{\prime \prime
}\left( b\right) \right\vert \right\} .%
\end{array}%
\end{equation*}
\end{remark}

\begin{remark}
In Theorem \ref{T4}, if we take $\eta \left( b,a\right) =b-a,\ \lambda =%
\frac{1}{2}\ $and $\alpha =1,\ $we have$;$%
\begin{equation*}
\begin{array}{l}
\left\vert \dfrac{f\left( a\right) +f\left( b\right) }{2}-\dfrac{1}{b-a}%
\dint\limits_{a}^{b}f(x)dx\right\vert \leq \dfrac{\pi \left( b-a\right) ^{2}%
}{64}\left\{ \left\vert f^{\prime \prime }\left( a\right) \right\vert
+\left\vert f^{\prime \prime }\left( b\right) \right\vert \right\} .%
\end{array}%
\end{equation*}
\end{remark}

\begin{theorem}
\label{T5} Let $A\subseteq 
\mathbb{R}
$ be an open invex subset with respect to $\eta :A\times A\rightarrow 
\mathbb{R}
\ $and $a,b\in A\ $with $a<a+\eta (b,a).\ $Suppose that $f:A\rightarrow 
\mathbb{R}
\ $is a differentiable function such that $f^{\prime \prime }\in L\left[
a,a+\eta (b,a)\right] .\ $If $\left\vert f^{\prime \prime }\right\vert ^{q}$
is $\lambda -$preinvex function on $\left[ a,a+\eta (b,a)\right] \ $for some
fixed $q>1$ then the following inequality holds:%
\begin{equation*}
\begin{array}{l}
\left\vert \dfrac{f\left( a\right) +f\left( a+\eta (b,a)\right) }{2}-\dfrac{%
\Gamma \left( \alpha +1\right) }{2\left( \eta (b,a)\right) ^{\alpha }}\left[
J_{a^{+}}^{\alpha }f\left( a+\eta (b,a)\right) +J_{\left( a+\eta
(b,a)\right) ^{-}}^{\alpha }f\left( a\right) \right] \right\vert  \\ 
\\ 
\leq \dfrac{\left( \eta (b,a)\right) ^{2}}{2\left( \alpha +1\right) }\left(
1-2^{-\alpha }\right) \dfrac{\pi }{4}\left( \left\vert f^{\prime \prime
}\left( a\right) \right\vert ^{q}+\dfrac{1-\lambda }{\lambda }\left\vert
f^{\prime \prime }\left( b\right) \right\vert ^{q}\right) ^{\frac{1}{q}}%
\end{array}%
\end{equation*}%
where $\alpha \in \left[ 0,1\right] \ $and $\frac{1}{p}+\frac{1}{q}=1.$
\end{theorem}

\begin{proof}
By using Definition \ref{D7}, Lemma \ref{L2} and H\"{o}lder's inequality we
have: 
\begin{equation*}
\begin{array}{l}
\left\vert \dfrac{f\left( a\right) +f\left( a+\eta (b,a)\right) }{2}-\dfrac{%
\Gamma \left( \alpha +1\right) }{2\left( \eta (b,a)\right) ^{\alpha }}\left[
J_{a^{+}}^{\alpha }f\left( a+\eta (b,a)\right) +J_{\left( a+\eta
(b,a)\right) ^{-}}^{\alpha }f\left( a\right) \right] \right\vert \\ 
\\ 
\leq \dfrac{\left( \eta (b,a)\right) ^{2}}{2}\int_{0}^{1}\left\vert \dfrac{%
1-\left( 1-t\right) ^{\alpha +1}-t^{\alpha +1}}{\alpha +1}\right\vert
\left\vert f^{\prime \prime }\left( a+\left( 1-t\right) \eta (b,a)\right)
\right\vert dt \\ 
\\ 
\leq \dfrac{\left( \eta (b,a)\right) ^{2}}{2\left( \alpha +1\right) }\left(
\int_{0}^{1}\left[ 1-\left( 1-t\right) ^{\alpha +1}-t^{\alpha +1}\right]
^{p}dt\right) ^{\frac{1}{p}}\left( \int_{0}^{1}\left\vert f^{\prime \prime
}\left( a+\left( 1-t\right) \eta (b,a)\right) \right\vert ^{q}dt\right) ^{%
\frac{1}{q}} \\ 
\\ 
\leq \dfrac{\left( \eta (b,a)\right) ^{2}}{2\left( \alpha +1\right) }\left(
\int_{0}^{1}\left[ 1-2^{-\alpha }\right] ^{p}dt\right) ^{\frac{1}{p}}\left(
\int_{0}^{1}\left( \dfrac{\sqrt{t}}{2\sqrt{1-t}}\left\vert f^{\prime \prime
}\left( a\right) \right\vert ^{q}+\dfrac{\left( 1-\lambda \right) \sqrt{1-t}%
}{2\lambda \sqrt{t}}\left\vert f^{\prime \prime }\left( b\right) \right\vert
^{q}\right) ^{q}dt\right) ^{\frac{1}{q}} \\ 
\\ 
\leq \dfrac{\left( \eta (b,a)\right) ^{2}}{2\left( \alpha +1\right) }\left(
1-2^{-\alpha }\right) \dfrac{\pi }{4}\left( \left\vert f^{\prime \prime
}\left( a\right) \right\vert ^{q}+\dfrac{1-\lambda }{\lambda }\left\vert
f^{\prime \prime }\left( b\right) \right\vert ^{q}\right) ^{\frac{1}{q}}.%
\end{array}%
\end{equation*}%
The proof is done.
\end{proof}

\begin{remark}
In Theorem \ref{T5}, if we take $\eta \left( b,a\right) =b-a,\ $we have$;$%
\begin{equation*}
\begin{array}{l}
\left\vert \dfrac{f\left( a\right) +f\left( b\right) }{2}-\dfrac{\Gamma
\left( \alpha +1\right) }{2\left( b-a\right) ^{\alpha }}\left[
J_{a^{+}}^{\alpha }f\left( b\right) +J_{b^{-}}^{\alpha }f\left( a\right) %
\right] \right\vert \\ 
\\ 
\leq \dfrac{\left( b-a\right) ^{2}}{2\left( \alpha +1\right) }\left(
1-2^{-\alpha }\right) \dfrac{\pi }{4}\left( \left\vert f^{\prime \prime
}\left( a\right) \right\vert ^{q}+\dfrac{1-\lambda }{\lambda }\left\vert
f^{\prime \prime }\left( b\right) \right\vert ^{q}\right) ^{\frac{1}{q}}.%
\end{array}%
\end{equation*}
\end{remark}

\begin{remark}
In Theorem \ref{T5}, if we take $\eta \left( b,a\right) =b-a\ $and $\alpha
=1,\ $we have$;$%
\begin{equation*}
\begin{array}{l}
\left\vert \dfrac{f\left( a\right) +f\left( b\right) }{2}-\dfrac{1}{b-a}%
\dint\limits_{a}^{b}f(x)dx\right\vert \leq \dfrac{\left( b-a\right) ^{2}}{8}%
\dfrac{\pi }{4}\left( \left\vert f^{\prime \prime }\left( a\right)
\right\vert ^{q}+\dfrac{1-\lambda }{\lambda }\left\vert f^{\prime \prime
}\left( b\right) \right\vert ^{q}\right) ^{\frac{1}{q}}.%
\end{array}%
\end{equation*}
\end{remark}

\begin{remark}
In Theorem \ref{T5}, if we take $\eta \left( b,a\right) =b-a,\ \lambda =%
\frac{1}{2}\ $and $\alpha =1,\ $we have$;$%
\begin{equation*}
\begin{array}{l}
\left\vert \dfrac{f\left( a\right) +f\left( b\right) }{2}-\dfrac{1}{b-a}%
\dint\limits_{a}^{b}f(x)dx\right\vert \leq \dfrac{\left( b-a\right) ^{2}}{8}%
\dfrac{\pi }{4}\left( \left\vert f^{\prime \prime }\left( a\right)
\right\vert ^{q}+\left\vert f^{\prime \prime }\left( b\right) \right\vert
^{q}\right) ^{\frac{1}{q}}.%
\end{array}%
\end{equation*}
\end{remark}

\begin{theorem}
\label{T6} Let $A\subseteq 
\mathbb{R}
$ be an open invex subset with respect to $\eta :A\times A\rightarrow 
\mathbb{R}
\ $and $a,b\in A\ $with $a<a+\eta (b,a).\ $Suppose that $f:A\rightarrow 
\mathbb{R}
\ $is a differentiable function such that $f^{\prime \prime }\in L\left[
a,a+\eta (b,a)\right] .\ $If $\left\vert f^{\prime \prime }\right\vert ^{q}$
is $\lambda -$preinvex function on $\left[ a,a+\eta (b,a)\right] \ $for some
fixed $q>1$ then the following inequality holds:%
\begin{equation*}
\begin{array}{l}
\left\vert \dfrac{f\left( a\right) +f\left( a+\eta (b,a)\right) }{2}-\dfrac{%
\Gamma \left( \alpha +1\right) }{2\left( \eta (b,a)\right) ^{\alpha }}\left[
J_{a^{+}}^{\alpha }f\left( a+\eta (b,a)\right) +J_{\left( a+\eta
(b,a)\right) ^{-}}^{\alpha }f\left( a\right) \right] \right\vert  \\ 
\\ 
\leq \dfrac{\left( \eta (b,a)\right) ^{2}}{2\left( \alpha +1\right) }\left( 
\dfrac{\alpha }{\alpha +2}\right) ^{1-\frac{1}{q}}\left( \dfrac{\pi }{2}-%
\dfrac{\sqrt{\pi }\Gamma \left( \alpha +\frac{3}{2}\right) }{\Gamma \left(
\alpha +2\right) }\right) ^{1/q}\left( \frac{\left\vert f^{\prime \prime
}\left( a\right) \right\vert ^{q}}{2}+\frac{1-\lambda }{\lambda }\frac{%
\left\vert f^{\prime \prime }\left( b\right) \right\vert ^{q}}{2}\right)
^{1/q}.%
\end{array}%
\end{equation*}%
where $\alpha \in \left[ 0,1\right] \ $and $\frac{1}{p}+\frac{1}{q}=1.$
\end{theorem}

\begin{proof}
By using Definition \ref{D7}, Lemma \ref{L2} and Power Mean's inequality, we
have:%
\begin{equation*}
\begin{array}{l}
\left\vert \dfrac{f\left( a\right) +f\left( a+\eta (b,a)\right) }{2}-\dfrac{%
\Gamma \left( \alpha +1\right) }{2\left( \eta (b,a)\right) ^{\alpha }}\left[
J_{a^{+}}^{\alpha }f\left( a+\eta (b,a)\right) +J_{\left( a+\eta
(b,a)\right) ^{-}}^{\alpha }f\left( a\right) \right] \right\vert \\ 
\\ 
\leq \dfrac{\left( \eta (b,a)\right) ^{2}}{2}\int_{0}^{1}\left\vert \dfrac{%
1-\left( 1-t\right) ^{\alpha +1}-t^{\alpha +1}}{\alpha +1}\right\vert
\left\vert f^{\prime \prime }\left( a+\left( 1-t\right) \eta (b,a)\right)
\right\vert dt \\ 
\\ 
\leq \dfrac{\left( \eta (b,a)\right) ^{2}}{2\left( \alpha +1\right) }\left(
\int_{0}^{1}\left\vert 1-\left( 1-t\right) ^{\alpha +1}-t^{\alpha
+1}\right\vert dt\right) ^{1-\frac{1}{q}} \\ 
\\ 
\times \left( \int_{0}^{1}\left\vert 1-\left( 1-t\right) ^{\alpha
+1}-t^{\alpha +1}\right\vert \left\vert f^{\prime \prime }\left( a+\left(
1-t\right) \eta (b,a)\right) \right\vert ^{q}dt\right) ^{\frac{1}{q}} \\ 
\\ 
\leq \dfrac{\left( \eta (b,a)\right) ^{2}}{2\left( \alpha +1\right) }\left(
\int_{0}^{1}\left[ 1-\left( 1-t\right) ^{\alpha +1}-t^{\alpha +1}\right]
dt\right) ^{^{1-\frac{1}{q}}} \\ 
\\ 
\times \left( \int_{0}^{1}\left[ 1-\left( 1-t\right) ^{\alpha +1}-t^{\alpha
+1}\right] \left( \frac{\sqrt{t}}{2\sqrt{1-t}}\left\vert f^{\prime \prime
}\left( a\right) \right\vert ^{q}+\frac{\left( 1-\lambda \right) \sqrt{1-t}}{%
2\lambda \sqrt{t}}\left\vert f^{\prime \prime }\left( b\right) \right\vert
^{q}\right) dt\right) ^{\frac{1}{q}}%
\end{array}%
\end{equation*}%
\begin{equation*}
\begin{array}{l}
\leq \frac{\left( \eta (b,a)\right) ^{2}}{2\left( \alpha +1\right) }\left( 
\dfrac{\alpha }{\alpha +2}\right) ^{1-\frac{1}{q}}\times \left( \frac{%
\left\vert f^{\prime \prime }\left( a\right) \right\vert ^{q}}{2}\int_{0}^{1}%
\left[ 1-\left( 1-t\right) ^{\alpha +1}-t^{\alpha +1}\right] \frac{\sqrt{t}}{%
\sqrt{1-t}}dt\right. \\ 
\\ 
\left. +\left( \frac{1-\lambda }{\lambda }\right) \frac{\left\vert f^{\prime
\prime }\left( b\right) \right\vert ^{q}}{2}\int_{0}^{1}\left[ 1-\left(
1-t\right) ^{\alpha +1}-t^{\alpha +1}\right] \frac{\sqrt{1-t}}{\sqrt{t}}%
dt\right) ^{\frac{1}{q}} \\ 
\\ 
\leq \dfrac{\left( \eta (b,a)\right) ^{2}}{2\left( \alpha +1\right) }\left( 
\dfrac{\alpha }{\alpha +2}\right) ^{1-\frac{1}{q}}\left( \dfrac{\pi }{2}-%
\dfrac{\sqrt{\pi }\Gamma \left( \alpha +\frac{3}{2}\right) }{\Gamma \left(
\alpha +2\right) }\right) ^{1/q}\left( \frac{\left\vert f^{\prime \prime
}\left( a\right) \right\vert ^{q}}{2}+\frac{1-\lambda }{\lambda }\frac{%
\left\vert f^{\prime \prime }\left( b\right) \right\vert ^{q}}{2}\right)
^{1/q}.%
\end{array}%
\end{equation*}%
The proof is done.
\end{proof}

\begin{remark}
In Theorem \ref{T6}, if we take $\eta \left( b,a\right) =b-a,\ $we have$;$%
\begin{equation*}
\begin{array}{l}
\left\vert \dfrac{f\left( a\right) +f\left( b\right) }{2}-\dfrac{\Gamma
\left( \alpha +1\right) }{2\left( b-a\right) ^{\alpha }}\left[
J_{a^{+}}^{\alpha }f\left( b\right) +J_{b^{-}}^{\alpha }f\left( a\right) %
\right] \right\vert \\ 
\\ 
\leq \dfrac{\left( b-a\right) ^{2}}{2\left( \alpha +1\right) }\left( \dfrac{%
\alpha }{\alpha +2}\right) ^{1-\frac{1}{q}}\left( \dfrac{\pi }{2}-\dfrac{%
\sqrt{\pi }\Gamma \left( \alpha +\frac{3}{2}\right) }{\Gamma \left( \alpha
+2\right) }\right) ^{1/q}\left( \frac{\left\vert f^{\prime \prime }\left(
a\right) \right\vert ^{q}}{2}+\frac{1-\lambda }{\lambda }\frac{\left\vert
f^{\prime \prime }\left( b\right) \right\vert ^{q}}{2}\right) ^{1/q}.%
\end{array}%
\end{equation*}
\end{remark}

\begin{remark}
In Theorem \ref{T6}, if we take $\eta \left( b,a\right) =b-a\ $and $\alpha
=1,\ $we have$;$%
\begin{equation*}
\begin{array}{l}
\left\vert \dfrac{f\left( a\right) +f\left( b\right) }{2}-\dfrac{1}{b-a}%
\dint\limits_{a}^{b}f(x)dx\right\vert \leq \dfrac{\left( b-a\right) ^{2}}{4}%
\left( \dfrac{1}{3}\right) ^{1-\frac{1}{q}}\left( \dfrac{\pi }{8}\right)
^{1/q}\left( \dfrac{\left\vert f^{\prime \prime }\left( a\right) \right\vert
^{q}}{2}+\dfrac{1-\lambda }{\lambda }\dfrac{\left\vert f^{\prime \prime
}\left( b\right) \right\vert ^{q}}{2}\right) ^{1/q}.%
\end{array}%
\end{equation*}
\end{remark}

\begin{remark}
In Theorem \ref{T6}, if we take $\eta \left( b,a\right) =b-a,\ \lambda =%
\frac{1}{2}\ $and $\alpha =1,\ $we have$;$%
\begin{equation*}
\begin{array}{l}
\left\vert \dfrac{f\left( a\right) +f\left( b\right) }{2}-\dfrac{1}{b-a}%
\dint\limits_{a}^{b}f(x)dx\right\vert \leq \dfrac{\left( b-a\right) ^{2}}{4}%
\left( \dfrac{1}{3}\right) ^{1-\frac{1}{q}}\left( \dfrac{\pi }{8}\right)
^{1/q}\left( \dfrac{\left\vert f^{\prime \prime }\left( a\right) \right\vert
^{q}}{2}+\dfrac{\left\vert f^{\prime \prime }\left( b\right) \right\vert ^{q}%
}{2}\right) ^{1/q}.%
\end{array}%
\end{equation*}%
\newpage
\end{remark}


\begin{thebibliography}{99}
\bibitem{Barani} Barani, A.; Ghazanfari, A.G.; Dragomir, S.S.:
Hermite-Hadamard inequality for functions whose derivatives absolute values
are preinvex. J. Inequal. Appl. 2012, 247 (2012)

\bibitem{Abramovitz} Abramowitz M, Stegun IA, editors. Handbook of
mathematical functions with formulas, graphs, and mathematical tables. New
York: Dover; 1965.

\bibitem{BM} A. Ben-Israel, B. Mond, What is invexity?, The Journal of the
Australian Mathematical Society, Series B-Applied Mathematics 28(1) (1986)
1-9.

\bibitem{D1} A. Barani, A.G. Ghazanfari, S.S. Dragomir, Hermite-Hadamard
inequality through prequsiinvex functions, RGMIA Research Report Collection
14 Article 64 (2011) 1-11.

\bibitem{Dragomir} S. S. Dragomir, R. P. Agarwal, Two inequalities for
diferentiable mappings and applications to special means of real numbers and
to trapezoidal formula, Appl. Math. Lett., 11 (1998), 91-95.

\bibitem{9} Di Donato, A. R., Jarnagin, M. P: The efficient calculation of
the incomplete beta-function ratio for half-integer values of the
parameters. Math. Comput. 21, 652-662 (1967)

\bibitem{14} Deng, J, Wang, J: Fractional Hermite-Hadamard inequalities for $%
\left( \alpha ,m\right) -$logarithmically convex
functions.J.Inequal.Appl.2013, Article ID 364$\left( 2013\right) $

\bibitem{imdat} \.{I}\c{s}can, \.{I}mdat. "Hermite-Hadamard's Inequalities
for Preinvex Function via Fractional Integrals and Related Fractional
Inequalities." American Journal of Mathematical Analysis 1.3 (2013): 33-38.

\bibitem{hanson} M.A. Hanson, On sufficiency of the Kuhn-Tucker conditions,
Journal of Mathematical Analysis and Applications 80 (1981) 545-550.

\bibitem{Noor} M.A. Noor and K.I. Noor, Generalized preinvex functions and
their properties. Journal of Appl. Math. Stochastic Anal., 2006(12736),
1--13, doi:10.1155/JAMSA/2006/12736

\bibitem{N1} M.A. Noor, Hadamard integral inequalities for product of two
preinvex function,Nonlinear Analysis Forum 14 (2009) 167-173.

\bibitem{N2} M.A. Noor, Some new classes of nonconvex functions, Nonlinear
Functional Analysis and its Applications 11 (2006) 165-171.

\bibitem{N3} M.A. Noor, On Hadamard integral inequalities involving two
log-preinvex functions, Journal of Inequalities in Pure and Applied
Mathematics 8(3) (2007) 1-14.

\bibitem{Pini} R. Pini, Invexity and generalized convexity, Optimization 22
(1991) 513-525.

\bibitem{3} Samko, S.G.; Kilbas, A.A.; Marichev, O.I.: Fractional Integrals
and Derivatives, Theory and Applications, Gordon and Breach, Yverdon,
Switzerland, 1993

\bibitem{12} Sarikaya, MZ, Set, E, Yaldiz, H, Ba\c{s}ak, N:
Hermite-Hadamard's inequalities for fractional integrals and related
fractional inequalities. Math.Comput. Model. 57, 2403-2407(2013)

\bibitem{sarikaya1} M. Zeki Sarikaya, Necmettin Alp and Hakan Bozkurt, On
Hermite-Hadamard Type Integral Inequalities for preinvex and log-preinvex
functions, Contemporary Analysis and Applied Mathematics, Vol.1, No.2,
237-252, 2013.

\bibitem{Sumeyye} S., Ermeydan and H., Y\i ld\i r\i m, Riemann-Liouville
Fractional Hermite-Hadamard Inequalities for differentiable $\lambda $%
-preinvex functions, (Accepted.)

\bibitem{10} Tun\c{c} M., and Yildirim, H., On MT-convexity,
http://arxiv.org/pdf/1205.5453.pdf.(2012).(preprint)

\bibitem{13} Wang, J, Li, X, Feckan, M, Zhou, Y: Hermite-Hadamard-type
inequalities for Riemann-Liouville fractional integrals via two kinds of
convexity. Appl. Anal. (2012).doi:10.1080/00036811.2012.727986
\end{thebibliography}
\end{document}